\documentclass[a4paper,reqno,11pt]{amsart}
\usepackage{amsmath,amsthm}
\usepackage{amssymb,latexsym}
\usepackage{enumerate}

\numberwithin{equation}{section}

\newtheorem{theorem}{Theorem}[section]
\newtheorem{remark}[theorem]{Remark}
\newtheorem{proposition}[theorem]{Proposition}
\newtheorem{lemma}[theorem]{Lemma}
\newtheorem{definition}[theorem]{Definition}
\newtheorem{exam}[theorem]{Example}
\newtheorem{corollary}[theorem]{Corollary}
\newtheorem{The main theorem}[theorem]{The main theorem}

\theoremstyle{definition}

\allowdisplaybreaks
\begin{document}

\title[Some results on Complex $m-$subharmonic classes]{Some results on Complex $m-$subharmonic classes}

\author{ Jawher Hbil and Mohamed Zaway}
\address{Department of Mathematics\\ Jouf University \\ P.O. Box: 2014, Sakaka, Saudi Arabia.}
\address{Department of mathematics, College of science, Shaqra University,
P.O. box 1040 Ad-Dwadimi 1191, Kingdom of Saudi Arabia.}

\address{Irescomath Laboratory, Gabes University, 6072 Zrig Gabes, Tunisia.}
\email{jmhbil@ju.edu.sa}
\email{m\_zaway@su.edu.sa}
\date{}

\maketitle


\renewcommand{\thefootnote}{}

\footnote{2010 \emph{Mathematics Subject Classification}: 32W20.}

\footnote{\emph{Key words and phrases}: $m-$subharmonic function, Capacity, Hessian operator., Convergence in m$-$capacity.}

\renewcommand{\thefootnote}{\arabic{footnote}}
\setcounter{footnote}{0}

\date{}

\begin{abstract}
In this paper we study the class $\mathcal{E}_{m}(\Omega)$ of $m-$subharmonic functions introduced by Lu in \cite{L1}. We prove that the convergence in $m-$capacity implies the convergence of the associated Hessian measure for functions that belong to $\mathcal{E}_{m}(\Omega)$. Then we extend those results to the class $\mathcal{E}_{m,\chi}(\Omega)$ that depends on a given increasing real function $\chi$. A complete characterization of those classes using the Hessian measure is given as well as   a subextension theorem relative to $\mathcal{E}_{m,\chi}(\Omega)$.
\end{abstract}

\maketitle

\section{Introduction}

In complex analysis, the  Monge-Ampere operator represents the objective of several studies since Bedford and Taylor \cite{BT82, BT76} demonstrated that the operator $ (dd ^ c.)^n$ is well defined on the set of locally bounded plurisubharmonic ( psh) functions defined on an hyperconvex domain $\Omega$ of $\mathbb{C}^n$. This domain was extended by Cegrell \cite{Ce1, Ce2} by introducing and investigating the classes $\mathcal{E}_0(\Omega)$, $\mathcal{F}(\Omega)$ and $\mathcal{E}(\Omega)$ that contain unbounded psh functions. He proved that  $\mathcal{E}(\Omega)$ is the largest domain of definition of the complex Monge-Ampere operator if we want the operator to be continuous for decreasing sequences. These works were taken up by Lu \cite{L1, L2} to define the complex Hessian operator $H_m$ on the set of $m-$subharmonic functions which coincides with the set of psh functions in the case $m = n$. By giving an analogy to Cegrell's classes, Lu studied some  analogous  classes denoted by $\mathcal{E}_{m}^0(\Omega)$, $\mathcal{F}_m(\Omega)$ and $\mathcal{E}_m(\Omega)$. One of the most well-known problems in this direction is the link between the convergence in capacity $Cap_m$ and the convergence of the complex Hessian operator.
The paper is organized as follows:
In section 2 we recall some preliminaries on the pluripotential theory for $m-$subharmonic function as well as the different energy classes which will be studied throughout the paper.\\
 In  section 3  we will be interested on giving a connection between  the convergence in capacity $Cap_m$ of a sequence of m-subharmonic functions $f_j$ toward $f$,  $\displaystyle liminf_{j} H_m (f_j)$ and  $H_m (f)$ when the function $f\in\mathcal E_m(\Omega)$. More precisely we prove the following theorem

\noindent
{\bf Theorem A. }

  If $(f_j)_j$ is a sequence of $m-$subharmonic function that belong to  $\mathcal E_m(\Omega)$ and satisfies $f_j\rightarrow f\in\mathcal E_m(\Omega)$ in $Cap_{m}$-capacity. Then
 $$ 1_{\{f>-\infty\}}H_m(f)\leq \liminf_{j\to+\infty} H_m(f_j).$$

As a consequence of Theorem A we obtain several results of convergence and especially we prove that if we modify the sufficient condition in the previous theorem, one may obtain the weak convergence of $H_m(f_j)$ to $H_m(f)$.\\

 In Section 4,  We will study the classes $\mathcal{E}_{m,\chi}(\Omega)$ introduced by Hung \cite{Hung} for a given increasing function $\chi$. Those classes generalized  the weighted pluricomplex energy classes investigated by Benelkourchi, Guedj and Zeriahi\cite{C-Z-Bel} and studied by \cite{Slimane, Slimane1, Hai}. We prove first the class $\mathcal{E}_{m,\chi}(\Omega)$ is fully included in the Cegrell class $\mathcal{E}_m(\Omega)$ and hence the Hessian operator $H_m(f)$ is well defined for every $f\in\mathcal{E}_{m,\chi}(\Omega)$. Then we will be interested on giving several results of the class $\mathcal{E}_{m,\chi}(\Omega)$ depending on some condition on the function $\chi$. Those results generalizes well know works in \cite{Slimane} and \cite{C-Z-Bel} it suffices to take $m=n$ to recover them. The most important result that we prove in this context is the given of a complete characterization for functions that belong to $\mathcal{E}_{m,\chi}(\Omega)$ using the class $\mathcal{N}_m(\Omega)$. In other words we show that
 $${\mathcal E}_{m,\chi}(\Omega)=\left\{ f \in {\mathcal N}_m(\Omega) \, / \,
\chi(f) \in L^1(H_m(f)) \right\}.$$
In the end we extend Theorem A to the class $\mathcal{E}_{m,\chi}(\Omega)$ by proofing the following result

{\bf Theorem B. }

 Let $\chi: \mathbb{R}^- \rightarrow \mathbb{R}^-$ be a continuous  increasing function such that $\chi(-\infty)>-\infty$ and $f, f_j\in \mathcal{E}_{m}(\Omega)$ for all $j\in \mathbb{N}$. Suppose that there is a function $g\in \mathcal{E}_{m}(\Omega)$ satisfying $f_j\geq g$ then:
  \begin{enumerate}
    \item If $f_j$ converges to $f$ in $Cap_{m-1}-$capacity then $\displaystyle\liminf_{j\rightarrow +\infty}-\chi(f_j)H_m(f_j)\geq -\chi(f)H_m(f).$
    \item If $f_j$ converges to $f$ in $Cap_{m}-$capacity then $-\chi(f_j)H_m(f_j)$ converges weakly to $-\chi(f)H_m(f).$
  \end{enumerate}

\section{Preliminaries}\label{1}
\subsection{m-subharmonic functions}
This section is devoted to recall some basic properties of $m-$subharmonic functions introduced by Blocki \cite{Bl1}. Those functions are admissible for the complex Hessian equation.
Throughout  this paper we denote by  $d:=\partial +\overline{\partial}$ ,$d^c:=i (\overline{\partial}-\partial)$ and  by $\Lambda_{p}(\Omega)$ the set of $(p,p)-$forms  in $\Omega$. The standard K$\ddot{a}$hler form defined on $\mathbb{C}^n$ will be denoted as $\beta:=dd^c|z|^2$.
\begin{definition}\cite{Bl1}\\

Let $\zeta\in\Lambda_{1}(\Omega)$  and $m\in \mathbb{N} \cap [1,n]$. The form $\zeta$ is called $m-$positive if it satisfies

$$\zeta^j\wedge\beta^{n-j}\geq 0, \ \ \forall j=1,\cdots,m$$
at every point of $\Omega$.

\end{definition}

\begin{definition}\cite{Bl1}\\

Let $\zeta\in\Lambda_{p}(\Omega)$ and $m\in \mathbb{N} \cap [p,n]$.
The $\zeta$ is  said to be $m-$positive on $\Omega$ if and only if the  measure
$$\zeta\wedge\beta^{n-m}\wedge\psi_1\wedge \cdots \wedge \psi_{m-p}$$is positive at  every point of $\Omega$
where $\psi_1,\cdots,\psi_{m-p}\in\Lambda_{1}(\Omega)$

\end{definition}
We will denote by $\Lambda_{p}^m(\Omega)$ the set of all $(p,p)-$forms on $\Omega$ that are $m-$positive.
In 2005, Blocki \cite{Bl1} introduced the notion of $m-$subharmonic functions and developed an analogous pluripotential theory. This notion is given in the following definition:
\begin{definition}
Let $f: \Omega\rightarrow \mathbb{R}\cup \{-\infty\}$. The function $f$ is called $m$-subharmonic if it satisfies the following:
\begin{enumerate}
  \item The function $f$ is subharmonic.
  \item For all  $\zeta_1, \cdots,\zeta_{m-1} \in \Lambda_{1}^m(\Omega)$ one has
  $$dd^cf\wedge \beta^{n-m}\wedge \zeta_1\wedge\cdots\wedge\zeta_{m-1}\geq 0$$

\end{enumerate}
\end{definition}
     We denote by $\mathcal{SH}_{m}(\Omega)$ the cone of  $m-$subharmonic functions defined on $\Omega$.
    \begin{remark}
  \begin{enumerate}
  In the case $m=n$ we have the following
   \item The  definition of $m-$positivity  coincides with  the classic definition of positivity given by Lelong for forms.
\item  The set  $\mathcal{SH}_{n}(\Omega)$ coincides with the set of psh functions on $\Omega$.
 \end{enumerate}
 \end{remark}

One can refer to \cite{Bl1}, \cite{SA}, \cite{HP} and \cite{L1} for more details about the properties of $m-$subharmonicity.\\
\begin{exam}
   \begin{enumerate}
     \item If $\zeta:=i( 4.dz_1\wedge d\overline{z}_1+4.dz_2\wedge d\overline{z}_2-dz_3\wedge d\overline{z}_3)$ then $\zeta\in \Lambda_{1}^2(\mathbb{C}^3)\setminus\Lambda_{1}^3(\mathbb{C}^3).$
     \item If  $f(z):=-|z_1|^2+2|z_2|^2+2|z_3|$ then  $f\in \mathcal{SH}_{2}(\mathbb{C}^3)\setminus \mathcal{SH}_{3}(\mathbb{C}^3)$.
     It is easy to see that $f \in \mathcal{SH}_{2}$. However, the restriction of $f$ on the line $( z_1, 0,0)$ is not subharmonic so $f$ is not a plurisubharmonic.
   \end{enumerate}
 \end{exam}

Following Bedford and Taylor \cite{BT76}, one can  define, by induction a closed nonnegative current when the function $f$ is  $m$-sh functions and   locally bounded as follows:
$$dd^cf_1\wedge\ldots\wedge dd^cf_k\wedge\beta^{n-m}:=dd^c(f_1dd^cf_2\wedge\ldots\wedge dd^cf_k\wedge\beta^{n-m}),$$
where $f_1,\ldots, f_k\in \mathcal{SH}_m(\Omega)\cap L_{loc}^{\infty}(\Omega).$
In particular, for a given $m-$sh function  $f\in \mathcal{SH}_m(\Omega)\cap L_{loc}^{\infty}(\Omega)$, we define the nonnegative Hessian measure of $f$ as follows
$$H_m(f)=(dd^cf)^m\wedge\beta^{n-m}.$$

\subsection{Cegrell classes of $m$-sh functions and $m-$capacity}
\begin{definition}
\begin{enumerate}
  \item A bounded domain $\Omega$  in $\mathbb{C}^{n}$ is said to be $m$-hyperconvex if the following property holds for some   continuous $m$-sh function $\rho:\; \Omega \rightarrow \mathbb{R}^{-}$:
 $$\{\rho<c\}\Subset\Omega,$$ for every $c<0.$
  \item A  set $M\subset \Omega$ is called $m-$polar if there exist $u\in \mathcal{SH}_m(\Omega)$ such that
    $$M\subset \{u=-\infty\}.$$
\end{enumerate}

\end{definition}
Throughout the rest of the paper, we denote by $\Omega$  a  $m$-hyperconvex domain of $\mathbb{C}^{n}$.
In \cite{L1} and \cite{L2}, Lu introduced the following classes of $m$-sh functions to generalize Cegrell's classes. We recall below the definitions of those classes.
\begin{definition}
We denote by:
  $$\mathcal{E}^0_m(\Omega)= \{f\in \mathcal{SH}_{m}^{-}(\Omega)\cap L^{\infty} (\Omega);\  \displaystyle\lim_{z\rightarrow \xi}f(z)=0\ \forall\xi\in\partial\Omega\  ,\ \displaystyle\int_{\Omega}H_m(f)<+\infty\},$$
     $$\mathcal{F}_m(\Omega)= \{f\in \mathcal{SH}_{m}^{-}(\Omega);\  \exists(f_j) \subset \mathcal{E}_m^0,\ f_j \searrow f\  in \ \Omega\, \ \displaystyle\sup_j\int_\Omega H_m(f_j)<+\infty\}.$$

and $$\mathcal{E}_m(\Omega)=\{f \in \mathcal{SH}_{m}^{-}(\Omega): \forall U\Subset\Omega,\exists\ f_U\in \mathcal{F}_m(\Omega);\   f_U=f\ on\ U\}.$$

\end{definition}

\begin{definition}
A function $f\in\mathcal{SH}_m(\Omega)$ is said to be $m$-maximal if for every  $g\in\mathcal{SH}_m(\Omega)$ such that if $g\leq f$ outside a compact subset of $\Omega$ then $g\leq f$ in $\Omega$.
\end{definition}
The previous notion represents an essential tool in the study of the Hessian operator since Blocki \cite{Bl1} showed that every  $m$-maximal function $f\in\mathcal{E}_{m}(\Omega)$ satisfies $H_{m}(f)=0.$
Take $(\Omega_j)_j$ a sequence of strictly $m$-pseudoconvex subsets of $\Omega$ such that $\Omega_j \Subset \Omega_{j+1}$, $\displaystyle\bigcup^{\infty}_{j=1}\Omega_j=\Omega$ and for every $j$ there exists a smooth strictly $m-$subharmonic function $\varphi$ in  a neighborhood $V$ of $\Omega_j$ such that $\Omega_j:= \{z\in V / \varphi(z) < 0\}$.
\begin{definition}\label{3}
Let $f \in \mathcal{SH}^{-}_{m}(\Omega)$ and  $(\Omega_j)_j$ be the  sequence defined above. Take  $f^j$
 the function defined by:
  $$f^j=\sup\left\{\psi \in \mathcal{SH}_{m}(\Omega):\; \psi_{|_{\Omega\setminus\Omega_j}}\leq f\right\}\in\mathcal{SH}_{m}(\Omega),$$
and define $\widetilde{f}:=(\displaystyle\lim_{j\rightarrow+\infty}f^j)^{*},$ called the smallest maximal $m$-subharmonic function majorant of $f$.
\end{definition}
It is clear that   $f\leq f^{j}\leq f^{j+1}$, so  $\displaystyle\lim_{j\rightarrow+\infty}f^j$ exists  on $\Omega$  except at an $m$-polar set, we deduce that $\widetilde{f}\in\mathcal{SH}_{m}(\Omega).$ Moreover, if $f\in \mathcal{E}_{m}(\Omega)$ then by  \cite{L2} and  \cite{Bl1} $\widetilde{f}\in\mathcal{E}_{m}(\Omega)$  and it is $m$-maximal on $\Omega.$ We denote  $\mathcal{MSH}_{m}(\Omega)$ is the family of $m$-maximal functions in $\mathcal{SH}_{m}(\Omega)$.\\
We cite below some useful properties of $\mathcal{MSH}_{m}(\Omega)$.
\begin{proposition}\cite{Bl1}
 Let $f, g \in\mathcal{E}_{m}(\Omega)$ and $\alpha\in \mathbb{R}$, $\alpha\geq0,$ then we have
 \begin{enumerate}
   \item $\widetilde{f+g}\geq \widetilde{f}+\widetilde{g}.$
   \item $\widetilde{\alpha f}=\alpha\widetilde{f}.$
   \item If $f\leq g$ then $\widetilde{f}\leq\widetilde{g}.$
   \item $\mathcal{E}_{m}(\Omega)\cap \mathcal{MSH}_{m}(\Omega)=\{f\in\mathcal{E}_{m}:\;\widetilde{f}=f\}.$
 \end{enumerate}
\end{proposition}
In \cite{Van}, author introduced a new Cegrell class $\mathcal{N}_m(\Omega):=\{f \in \mathcal{E}_{m}:\; \widetilde{f}=0\}.$
It is easy to check that  $\mathcal{N}_m(\Omega)$ is a convex cone satisfying
$$\mathcal{E}^{0}_m(\Omega)\subset \mathcal{F}_m(\Omega)\subset \mathcal{N}_m(\Omega)\subset \mathcal{E}_m(\Omega).$$
\begin{definition}
Let $\mathcal{L}_m \in \{\mathcal{F}_m, \mathcal{N}_m, \mathcal{E}_{m}\}$.
 We define
  $$\mathcal{L}^a_m(\Omega):=\{f \in \mathcal{L}_{m}:\; H_{m}(f)(P)=0,\; \forall P \;m\hbox{-polar\; set}\}.$$
  \end{definition}
\begin{definition}
\begin{enumerate}
  \item Let $E$ be a  Borel subset of $\Omega$. The $Cap_s$-capacity of a $E$ with respect to $\Omega$ is given as follows:
$$Cap_{s}(E)=Cap_{s}(E,\Omega)=\sup\left\{\int_{E}H_s(f)\;,\;f\in \mathcal{SH}_{m}(\Omega), -1\leq f\leq 0 \right\}$$
where $1\leq s\leq m$.
  \item We say that a sequence $(f_j)_j$, of real-valued borel measurable functions  defined on $\Omega$,  converges to $f$ in $Cap_{s}$-capacity, when $j\rightarrow+\infty$ if for every compact subset $K$ of $\Omega$ and  $\varepsilon> 0$ the following limit holds
$$\displaystyle\lim_{j\rightarrow+\infty}Cap_{s}(\{z\in K: |f_j(z)-f(z)|>\varepsilon\})=0.$$
\item For a given Borel subset $E\subset \Omega$, the outer $s-$capacity  $Cap_{s}^{\star}$ of $E$  is defined as
$$Cap_{s}^{\star}(E,\Omega):=inf \{Cap_{s}(F,\Omega);\ E\subset F\ and\ F\ is\ an\ open\ subset\ of\ \Omega\}.$$
  \end{enumerate}
\end{definition}
\begin{remark}
  For a given subset $E$ of $\Omega$ one can defined $h_{E,\Omega}$ as follows
  $$h_{E,\Omega}:=\sup\{f(z); f\in \mathcal{SH}^-(\Omega): f\leq -1\ on\ E\}.$$
  Using the definitions above and Theorem 2.20 in \cite{L1}, we have the following
  $$Cap_{m}^{\star}(E,\Omega)=\int_{\Omega}H_m(h_{E,\Omega}^{*})$$
  where $h_{E,\Omega}^{*}$ is the smallest upper semicontinuous function majorant of $h_{E,\Omega}$.
\end{remark}
\section{Convergence in $Cap_m-$Capacity}

\begin{proposition}( See \label{01}\cite{HP} and  \cite{A})
 \begin{enumerate}
   \item For every $f, g\in\mathcal E_m(\Omega)$, such that $g\leq f$ one has
   $$1_{\{f=-\infty\}}H_m(f)\leq1_{\{g=-\infty\}}H_m(g)$$
   \item  If $f\in\mathcal E_m(\Omega)$, and $g\in\mathcal E_m^a(\Omega)$ then
    $$1_{\{f+g=-\infty\}}H_m(f+g)\leq1_{\{f=-\infty\}}H_m(f)$$
 \end{enumerate}
\end{proposition}
\begin{proposition}\label{02}
 For every non-negative measures $\mu$, $\nu$  on $\Omega$, satisfying $(\mu + \nu)(\Omega) < \infty$ and $\int_{\Omega}-f d\mu\geq\int_{\Omega}-f d\nu$  for all $f\in\mathcal E_m^0(\Omega)$, one has $\mu(K) \geq\nu(K)$ for all complete $m-$polar subsets $K$ in $\Omega$.
\end{proposition}
\begin{proof}
Using Theorem 1.7.1 in \cite{L2}, we get
$$\int_{\Omega}-f d\mu\geq\int_{\Omega}-f d\nu \  \forall f\in \mathcal{SH}_{m}^{-}(\Omega)\cap L^{\infty}(\Omega).$$
Take  $g \in \mathcal{SH}_{m}^{-}(\Omega)$  such that $K = \{g = -\infty\}$, then for all $\varepsilon > 0$, we have
$$\int_{\Omega}-\max(\varepsilon g,-1)  d\mu\geq\int_{\Omega}-\max(\varepsilon g,-1)  d\nu.$$
The result follows by letting $\varepsilon \to 0$.
\end{proof}

We consider the sets $\mathcal{P}_{m}(\Omega)$ and  $\mathcal{Q}_{m}(\Omega)$ defined as follows:
$$\mathcal{P}_{m}(\Omega)=\{f\in\mathcal E_m(\Omega)\ ;\exists\   P_1,..., P_n \ polar \ in \ \mathbb{C}\ /\ 1_{\{f=-\infty\}}H_m(f)(\Omega\backslash P_1\times...\times P_n)=0\}.$$
$$\mathcal{Q}_{m}(\Omega)=\{(f,g)\in (\mathcal E_m(\Omega))^2; \  \forall z\in \Omega , \exists V\ \in \mathcal{V}(z)\ and\  u_V\in \mathcal E^a_m(V)\ /\ f+u_V\leq g\ on\ V  \}.$$
We cite below some properties of the class $\mathcal{P}_{m}(\Omega)$ that will be useful further
\begin{proposition}\label{03}
 \begin{enumerate}
   \item If $f\in\mathcal{SH}_{m}^{-}(\Omega)$, $g\in\mathcal{P}_{m}(\Omega)$ and $f\geq g$ then $f\in\mathcal{P}_{m}(\Omega)$.
  \item If $f, g\in\mathcal{P}_{m}(\Omega)$  then $f+g\in\mathcal{P}_{m}(\Omega)$.
 \end{enumerate}
\end{proposition}
\begin{proof}
(1) Since $f\in  \mathcal E_m(\Omega)$ so is $g$. Now assume that there exists $ P_1,..., P_n \ polar \ in \ \mathbb{C}$ such that $1_{\{g=-\infty\}}H_m(g)(\Omega\backslash P_1\times...\times P_n)=0$. Then by proposition \ref{01}, we deduce that
$$1_{\{f=-\infty\}}H_m(f)(\Omega\backslash P_1\times...\times P_n)=0.$$
It follows that $f\in\mathcal{P}_{m}(\Omega)$. The proof of the first assertion is completed.\\
(2) Using \cite{L2}, the set $\mathcal{E}_m(\Omega)$ is a convex cone. Hence if  $f,g\in  \mathcal E_m(\Omega)$ so is $f+g$.
Take $ P_1,..., P_n \ polar \ in \ \mathbb{C}$ such that $1_{\{g=-\infty\}}H_m(g)(\Omega\backslash P_1\times...\times P_n)=0$. We have
$$H_m(f+g)=\displaystyle\sum_{k=0}^{m}\binom{m}{k}(dd^c f)^k\wedge (dd^cg)^{m-k}\wedge \beta^{n-m}.$$
If we fix $k\in\{0,...,m\}$ then by lemma 1 in \cite{Hai} we obtain the following writing
$$(dd^c f)^k\wedge (dd^cg)^{m-k}\wedge \beta^{n-m}=\mu+\mathbf{1}_{\{f=g=-\infty\}}(dd^c f)^k\wedge (dd^cg)^{m-k}\wedge \beta^{n-m}$$
where $\mu$ is a measure that has no mass on $m-$polar sets. We deduce that
$$\begin{array}{lcl}
 \mathbf{1}_{\{f+g=-\infty\}}H_m(f+g) & = & \displaystyle\sum_{k=0}^{m}\binom{m}{k}\mathbf{1}_{\{f=g=-\infty\}}(dd^c f)^k\wedge (dd^cg)^{m-k}\wedge \beta^{n-m}.
 \end{array}$$
 It follows by Lemma 5.6 in \cite{HP} that
 $$\begin{array}{l}
  \displaystyle\int_{\Omega\setminus(P_1\times...\times P_n)}\mathbf{1}_{\{f+g=-\infty\}}H_m(f+g)\\
    =\displaystyle\sum_{k=0}^{m}\binom{m}{k}\displaystyle\int_{\Omega\setminus(P_1\times...\times P_n)}\mathbf{1}_{\{f=g=-\infty\}}(dd^c f)^k\wedge (dd^cg)^{m-k}\wedge \beta^{n-m} \\\
   \leq 2^m \left(\displaystyle\int_{\Omega\setminus(P_1\times...\times P_n)\cap{\{f=g=-\infty\}}}H_m(f)\right)^{\frac{1}{m}}.\left(\displaystyle\int_{\Omega\setminus(P_1\times...\times P_n)\cap{\{f=g=-\infty\}}}H_m(g)\right)^{\frac{1}{m}}\\
   =0.
\end{array}
$$
We conclude that $f+g\in\mathcal{P}_{m}(\Omega)$.
\end{proof}
   The following theorem represents the first main result in this paper.
\begin{theorem}\label{04}
  If $f_j$ is a sequence of $m-$subharmonic function that belong to  $\mathcal E_m(\Omega)$  and satisfies $f_j\rightarrow f\in\mathcal E_m(\Omega)$ in $Cap_{m}$-capacity. Then
 $$ 1_{\{f>-\infty\}}H_m(f)\leq \liminf_{j\to+\infty} H_m(f_j).$$
\end{theorem}
\begin{proof}
Take  $ 0\leq \varphi\in C_0^\infty(\Omega)$ and $\Omega_1\Subset\Omega$ such that  $supp f\Subset\Omega_1$. it suffices to show that
$$\liminf_{j\to+\infty}\int_{\Omega}  \varphi  H_m(f_j)\geq \int_{\Omega}1_{\{f>-\infty\}}\varphi H_m(f).$$
For each $a > 0$ one has that
$$\int_{\Omega}  \varphi   H_m(f_j)-\int_{\Omega}1_{\{f>-\infty\}}\varphi H_m(f)=A_1+A_2+A_3,$$
where
\begin{align*}
 A_1& = \int_{\Omega}  \varphi  \left(H_m(f_j)-H_m(\max(f_j,-a))\right)+\int_{\Omega}1_{\{f=-\infty\}}\varphi H_m(f)
\\ A_2& = \int_{\Omega}  \varphi  \left(H_m(\max(f_j,-a))-H_m(\max(f,-a))\right)
\\ A_3& = \int_{\Omega}  \varphi  \left(H_m(\max(f,-a))-H_m(f)\right).
\end{align*}
Using Theorem 3.6 in \cite{HP} we obtain that
\begin{align*}
 A_1& = \int_{\{f_j\leq-a\}}    \varphi(H_m(f_j)-H_m(\max(f_j,-a)))+\int_{\Omega}1_{\{f=-\infty\}}\varphi H_m(f)
\\ &\geq -\int_{\{f_j\leq-a\}}   \varphi  H_m(\max(f_j,-a))+\int_{\Omega}1_{\{f=-\infty\}}\varphi H_m(f)
\\& \geq -\int_{\{f_j\leq-a\}\cap\{|f_j-f|\leq1\}}   \varphi  H_m(\max(f_j,-a))-\int_{\{|f_j-f|>1\}}  \varphi  H_m(\max(f_j,-a))
\\&+\int_{\Omega}1_{\{f=-\infty\}}\varphi H_m(f)
\\& \geq -\int_{\{f<-a+2\}}  \varphi H_m(\max(f_j,-a))-a^nCap_{m}(\{|f_j-f|>1\}\cap\Omega_1)
\\&+\int_{\Omega}1_{\{f=-\infty\}}\varphi H_m(f)
\\& \geq \int_\Omega h_{\{f<-a+2\}\cap\Omega_1,\Omega}   \varphi H_m(\max(f_j,-a))-a^nCap_{m}(\{|f_j-f|>1\}\cap\Omega_1)
\\&+\int_{\Omega}1_{\{f=-\infty\}}\varphi H_m(f).
\end{align*}
If we let  $j\to +\infty$ then by Theorem 3.8 in \cite{HP} we obtain
$$ \liminf_{j\to+\infty}A_1\geq\int_\Omega h_{\{f<-a+2\}\cap\Omega_1,\Omega}   f H_m(\max(f_j,-a))+\int_{\Omega}1_{\{f=-\infty\}}fH_m(f).$$
It follows by Theorem 3.8 in \cite{HP} that for all $s>0$ one has
\begin{align*}
\liminf_{a\to+\infty}(\liminf_{j\to+\infty}A_1)& \geq\liminf_{a\to+\infty}\int_\Omega h_{\{f<-a+2\}\cap\Omega_1,\Omega}   \varphi H_m(\max(f_j,-a))+\int_{\Omega}1_{\{f=-\infty\}}\varphi H_m(f)
\\ & \geq\liminf_{a\to+\infty}\int_\Omega h_{\{f<-s\}\cap\Omega_1,\Omega}   \varphi H_m(\max(f_j,-a))+\int_{\Omega}1_{\{f=-\infty\}}\varphi H_m(f))
\\ & =\int_\Omega h_{\{f<-s\}\cap\Omega_1,\Omega}   \varphi H_m(f)+\int_{\Omega}1_{\{f=-\infty\}}\varphi H_m(f).
\end{align*}
Since $\displaystyle\lim_{s\rightarrow +\infty}Cap_m(\{f < -s\} \cap\Omega_1)=0$ then there exists a subset $A$ of $\Omega$ with $Cap_{m}(A) = 0$ such that the function  $ h_{\{f<-s\}\cap\Omega_1,\Omega}$ increases to $0$ as $s\rightarrow +\infty$ on $\Omega\backslash A$. Now  by a
decomposition theorem  in \cite{L2} we get that if  $s \to +\infty$

$$\liminf_{a\to+\infty}(\liminf_{j\to+\infty}A_1) \geq\int_\Omega -1_E \varphi H_m(f)+\int_{\Omega}1_{\{f=-\infty\}}\varphi H_m(f)\geq0.$$

It follows  by Theorem 3.8 in \cite{HP} that
\begin{align*}
\liminf_{j\to+\infty}&\left( \int_\Omega  \varphi H_m(f_j)-\int_{\Omega}1_{\{f>-\infty\}}\varphi H_m(f)\right)
\\&\geq\liminf_{a\to+\infty}\liminf_{j\to+\infty}A_1+\liminf_{a\to+\infty}A_3\geq0.
\end{align*}

\end{proof}
\begin{corollary}\label{05}
  Let $(f_j)_j\subset \mathcal E_m(\Omega)$ such that  $f_j\rightarrow f\in\mathcal E_m(\Omega)$ in $Cap_{m}$-capacity. If $(f_j,f)\in \mathcal Q_m(\Omega)$
  for all $j\geq1$. Then
 $$ H_m(f)\leq \liminf_{j\to+\infty} H_m(f_j).$$
\end{corollary}
\begin{proof}
  By combining the Definition of $\mathcal{Q}_m(\Omega)$ and the proposition \ref{01} we get that
$$ 1_{\{f=-\infty\}}H_m(f)\leq 1_{\{f_j=-\infty\}}H_m(f_j)\leq H_m(f_j).$$
The result follows using Theorem \ref{04}.
\end{proof}
\begin{corollary}\label{06}
  Let $(f_j)_j\subset \mathcal F_m(\Omega)$ such that $f_j\rightarrow f\in\mathcal F_m(\Omega)$ in $Cap_{m}$-capacity. If $(f_j,f)\in \mathcal Q_m(\Omega)$
  for all $j\geq1$. and
 $$ \lim_{j\to+\infty}\int_\Omega H_m(f_j)=\int_\Omega H_m(f).$$
 Then $ H_m(f_j)\to H_m(f)$ weakly as $j \to+\infty$.
\end{corollary}
\begin{proof}
  Without loss of generality one can assume that $H_m(f_j)\to \mu$ weakly as $j \to+\infty$. Using Corollary \ref{05} we obtain that $\mu \geq H_m(f)$. On the other hand,
   $$ \mu(\Omega)\leq\liminf_{j\to+\infty}\int_\Omega H_m(f_j)=\int_\Omega H_m(f).$$
Hence $\mu=H_m(f)$.
\end{proof}

\begin{theorem}\label{07}
  Let $f_j, g\in\mathcal E_m(\Omega)$, $f \in\mathcal P_m(\Omega)$, and $D\Subset\Omega$. Assume that
   \begin{itemize}
     \item $f_j\rightarrow f$ in $Cap_{m}$-capacity.
     \item For all $j\geq1$, $f_j\geq g$ on $\Omega\backslash D$.
   \end{itemize}

 Then
 $H_m(f_j)\rightarrow H_m(f)$ weakly as $j\rightarrow\infty.$
\end{theorem}
\begin{proof}
As $f \in\mathcal P_m(\Omega)$ there exist $P_1,..., P_n$ be $m-$polar subsets in $\mathbb{C}$ such that
$$1_{\{f=-\infty\}}H_m(f)(\Omega\backslash P_1\times...\times P_n)=0.$$
Take
$$\tilde{f}_j=\max(f_j, g), \ \ \   \tilde{f}=\max(f, g)$$
It easy to check that  $\tilde{f}_j, f \in \mathcal{E}_m(\Omega)$ and $\tilde{f}_j\rightarrow \tilde{f}$ in $Cap_{m}$-capacity. Moreover $\tilde{f}_j |_{\Omega\backslash D} = f_j|_{\Omega\backslash D}$ and  $\tilde{f}|_{\Omega\backslash D}=f|_{\Omega\backslash D}$. Using Theorem 3.8 in \cite{HP},  we get that  $H_m(\tilde{f_j})\rightarrow H_m(\tilde{f})$ weakly as $j\rightarrow\infty$. Let $\Omega_1$ be a $m-$hyperconvex domain such that $D \Subset\Omega_1 \Subset \Omega$. By Stokes’ theorem we have
$$\limsup_{j\to+\infty} \int_{\Omega_1}   H_m(f_j)=\limsup_{j\to+\infty} \int_{\Omega_1}   H_m(\tilde{f}_j)\leq\int_{\bar{\Omega}_1}   H_m(\tilde{f})<\infty.$$
Hence without loss of generality one may assume that there exists a positive measure $\mu$ such that $ H_m(f_j)\rightarrow \mu$ weakly as $j\rightarrow\infty.$ The proof will be completed if we show that $\mu = H_m(f)$ on $\Omega_1$. For this take $u\in\mathcal E_m^0(\Omega_1)$, then by Stokes’ theorem we obtain that

$$\int_{\Omega_1}   -u d\mu=\lim_{j\to+\infty} \int_{\Omega_1}  -u  H_m(f_j)\geq\lim_{j\to+\infty} \int_{\Omega_1}  -u  H_m(\tilde{f}_j)\geq\lim_{j\to+\infty} \int_{\Omega_1}  -u  H_m(\tilde{f}).$$
Moreover by Proposition \ref{02} and \cite{Hi} we get
$$ H_m(f)(K)\leq \mu(K).\qquad (*)$$
for all compact subsets $K$ of $E_1,..., E_n$. We deduce that  $\mu \geq 1_{\{f=-\infty\}}H_m(f)$. So by Theorem \ref{04} we obtain
$$H_m(f)\leq \mu \ on \ \Omega_1.$$
Now let $\Omega_2$ be a domain satisfying $D \Subset\Omega_2 \Subset \Omega_1$. By Stokes theorem we obtain that

\begin{align*}
\mu(\Omega_2)& \leq \liminf_{j\to+\infty} \int_{\Omega_2}   H_m(f_j)=\liminf_{j\to+\infty} \int_{\Omega_2}   H_m(\tilde{f}_j)
\\ &\leq  \int_{\bar{\Omega}_2}   H_m(\tilde{f})\leq \int_{\Omega_1}   H_m(\tilde{f})=\int_{\Omega_1}   H_m(f).
\end{align*}
It follows that
$$\mu(\Omega_1)\leq H_m(f)(\Omega_1).\qquad (**)$$
Using $(*)$ and $(**)$ we deduce that  $\mu= H_m(f)$
 on $\Omega_1$.
\end{proof}
The following lemma will be useful in the proof of several results in this paper.
\begin{lemma}\label{08}
Fix $f \in {\mathcal F}_m(\Omega)$. Then  for all $s>0$ and $ t > 0$, one has

\begin{equation}
 t^m
Cap_m(f\ <-s-t ) \leq \int_{\{f< -s\}} H_m(f) \leq s^m
Cap_m(f < -s ).
\end{equation}
\end{lemma}

 \begin{proof} Let $t, \ s>0$ and $K$ be a compact subset satisfying
  $K \subset \{ f < -s -t \} $. We have
\begin{multline*}
Cap_m (K) = \int _\Omega  H_m(h_K^*) =
 \int _{ \{ f  < -s -t \}}
 H_m( h_K^* ) \\
= \int _{\{ f  < -s + th_K^*  \}}
H_m(h_K^*) = \frac{1}{t^m}\int _
{ \{ f  < g \}}
 H_m( g),
\end{multline*}
Using Theorem 3.6 in \cite{HP} we obtain that
\begin{multline*}
\frac{1}{t^m} \int _{ \{ f  < g \}}
 H_m(g) = \frac{1}{t^m}\int_
{ \{ f  < \max( f , g) \}}
 H_m(\max ( f , g))
  \leq\\
 \frac{1}{t^m}\int _{ \{ f  < \max ( f,  g) \}}
 H_m(f ) =    \frac{1}{t^m}\int _{
 \{ f  <  -s  + t h_K  \}}
 H_m(f)  \leq  \frac{1}{t^m}\int _{
 \{ f  <  -s   \}}
 H_m( f).
\end{multline*}
The left hand inequality of $(3.1)$ follows by taking the supremum over all compact sets $K\subset \Omega$.\\
 For the right hand inequality, we have
\begin{multline*}\int _{
 \{ f  \leq   -s   \}}
 H_m( f)  = \int _{\Omega}
H_m( f)   - \int _{f>-s }
H_m( f)  \\
 = \int _{\Omega}
 H_m(\max( f , -s )) - \int _{f>-s }
H_m(\max( f , -s ) )\\ = \int _{f \leq -s }
H_m( \max( f , -s )) \leq s^m Cap_m \{f \leq -s \}.
\end{multline*}
The result follows.
\end{proof}
\begin{remark}
  Using the previous lemma we deduce the following results
  \begin{enumerate}
    \item $f\in {\mathcal F}_m(\Omega)$ if and only if
$\displaystyle\limsup_{s\to 0} s^m Cap_m (\{f < -s\})<+\infty.$
    \item If $f\in {\mathcal F}_m(\Omega)$ then
$$
 \int_\Omega H_m(f ) = \lim _{s\to 0} s^m Cap_m (\{f < -s\})
 $$
and
 $$
 \int_{\{f=-\infty\}}H_m( f ) = \lim _{s\to +\infty } s^mCap_m (\{f <-s\}).
 $$
\item The function $f\in {\mathcal F}_m^a(\Omega) $ if and only if $\displaystyle\lim _{s\to +\infty } s^n Cap_m (\{f < -s\})=0.$ Indeed it is known  that if $f$ is an  m$-$sh function on $\Omega$ then $H_m(f)(P)=0$ for every m$-$polar set $P\subset \Omega$ if and only if $H_m(f)(\{f=-\infty\})=0$ which follows directly from the previous assertion of this remark.
 \end{enumerate}

\end{remark}
\section{The Class $\mathcal{E}_{m,\chi}(\Omega)$ }
Throughout this section   $\chi : \mathbb{R}^-  \to \mathbb{R}^- $ will be an increasing function. In \cite{Hung} Hung introduced the class $\mathcal{E}_{m,\chi}(\Omega)$ to  generalize the fundamental weighted energy classes introduced firstly by Benelkourchi,  Guedj, and  Zeriahi \cite{C-Z-Bel}. Such class is defined as follows:
\begin{definition}
  We say that $f\in \mathcal{E}_{m,\chi}(\Omega)$ if and only if there exits  $(f_j)_j\subset \mathcal{E}^0_{m}(\Omega)$ such that $f_j\searrow f$ in $\Omega$ and
  $$\sup_{j\in \mathbb{N}}\int_{\Omega}(-\chi(f_j))H_m(f_j)<+\infty.$$
\end{definition}
\begin{remark}
 It is clear that the class $\mathcal{E}_{m,\chi}(\Omega)$ generalizes all analogous Cegrell classes defined by Lu in \cite{L1} and \cite{L2}. Indeed
\begin{enumerate}
\item ${\mathcal E}_{m,\chi} (\Omega )={\mathcal F}_m(\Omega)$ when $\chi(0)
 \neq 0$ and $\chi$ is bounded.
\item ${\mathcal E}_{m,\chi} (\Omega )={\mathcal E}_m^p(\Omega)$ in the case when $\chi(t)=-(-t)^p$;
\item ${\mathcal E}_{m,\chi} (\Omega)={\mathcal F}_m^p(\Omega)$ in the case when  $\chi(t)=-1-(-t)^p$.

\end{enumerate}
Note that if we take $m=n$ in all the previous cases we recover the classic Cegrell classes defined in \cite{Ce1} and \cite{Ce2}.
\end{remark}
Note that in the case $\chi (0) \neq 0 $ one has that ${\mathcal E}_{m,\chi} (\Omega)\subset{\mathcal F}_m(\Omega)$ so the Hessian operator is well defined in and is with finite total mass on $\Omega$. So in the rest of this paper we will always consider the case $\chi (0)=0 $.\\

In the following Theorem we will prove that the Hessian operator is well defined on $\mathcal E_{m,\chi}(\Omega)$. Note that this result was proved in \cite{Hung} but with an extra condition ($\chi(2t)\leq a.\chi(t)$). Here we omit that condition and the proof of such  result is completely different.
\begin{theorem}\label{10}
Assume that $\chi\not\equiv 0$.
Then
$$\mathcal E_{m,\chi}(\Omega)\subset\mathcal E_m(\Omega).$$
So for every $f \in\mathcal E_{m,\chi}(\Omega)$,
$H_m(f)$ is well defined and $-\chi(f) \in L^1(H_m(f))$.

\end{theorem}
\begin{proof}
Since  $\chi\not\equiv 0$ so there exists  $t_0 > 0$   such that $\chi(-t_0) < 0$. Take $\chi_1$ an increasing
function satisfying $\chi_1' = \chi_1'' = 0$ on $[-t_0, 0]$, $\chi_1$ is convex on $]-\infty,-t_0]$ and $\chi_1 \geq\chi$.
Let $g \in \mathcal{SH}_{m}^{-}(\Omega)$, then

$$dd^c\chi_1(g)\wedge\beta^{n-m}=\chi_1''(g) dg\wedge d^cg\wedge\beta^{n-m}+\chi_1'(g) dd^c\chi_1(g)\wedge\beta^{n-m}\geq0.$$
So the function $\chi_1(g) \in \mathcal{SH}_{m}^{-}(\Omega)$.
Now consider $f \in\mathcal E_{m,\chi}(\Omega)$, then by definition there exists a sequence $f_j \in\mathcal E_{m}^0(\Omega)$ that decreases to $f$ and
satisfying
$$\sup_{j\in \mathbb{N}}\int_{\Omega}-\chi(f_j)H_m(f_j)<\infty.$$
By definition of the class $\mathcal{E}_{m}(\Omega)$, it remains to prove that $f$ coincides locally with a function in $\mathcal{F}_{m}(\Omega)$. For this take $G \Subset\Omega$ be a domain and
consider the function
$$f_j^G:=\sup\{g\in\mathcal{SH}_{m}^{-}(\Omega); g\leq f_j\  on\  G\}.$$
We have $f_j^G\in \mathcal E_{m}^0(\Omega)$ and $f_j^G\searrow f$
 on $G$. Take $\varphi \in \mathcal E_{m}^0(\Omega)$ such that $\chi_1( f_1)\leq \varphi.$
We obtain using integration by parts that
\begin{align*}
\sup_{j\in \mathbb{N}}\int_{\Omega}-\varphi H_m(f_j^G)& \leq\sup_{j\in \mathbb{N}}\int_{\Omega}-\varphi H_m(f_j)
\\ &  \leq\sup_{j\in \mathbb{N}}\int_{\Omega}-\chi_1(f_1)H_m(f_j)
\\& \leq\sup_{j\in \mathbb{N}}\int_{\Omega}-\chi_1(f_j)H_m(f_j)
\\ &  \leq\sup_{j\in \mathbb{N}}\int_{\Omega}-\chi(f_j)H_m(f_j)<\infty.
\end{align*}

We deduce that

$$\sup_{j\in \mathbb{N}}\int_{\Omega} H_m(f_j^G) \leq (-\sup_G \varphi)^{-1}\sup_{j\in \mathbb{N}}\int_{\Omega}-\varphi H_m(f_j^G)<\infty.$$

It Follows that  the limit $\displaystyle\lim_{j\rightarrow +\infty} f_j^G\in \mathcal F_{m}(\Omega)$ and therefore $f\in \mathcal E_{m}(\Omega)$.\\

For the second assertion, we have that every  $f\in \mathcal E_{m,\chi}(\Omega)$ is upper semicontinuous, so the sequence of  measures $\mu_j:=-\chi(f_j)H_m(f_j)$ is bounded. Take $\mu$ a cluster point of $\mu_j$ then  $-\chi(f)H_m(f)\leq \mu$. Hence $\int_{\Omega}-\chi(f)H_m(f)<\infty$ and the desired result follows.
\end{proof}

\begin{proposition}
 Then the following statements are equivalent:
\begin{enumerate}
  \item $\chi(-\infty)=-\infty$
  \item $\mathcal{E}_{m,\chi}(\Omega)\subset\mathcal{E}_m^a(\Omega)$.
\end{enumerate}
\end{proposition}
\begin{proof}
We will prove that $(1)\Rightarrow (2)$. For this assume  that $\chi(-\infty)=-\infty$ and take $f\in\mathcal{E}_{m,\chi}(\Omega)$. By definition of the class $\mathcal{E}_{m,\chi}(\Omega)$, there exists a sequence $\{f_j\}\subset\mathcal{E}_m^0$ such that $f_j\searrow f$ and
$$ \sup\limits_{j}\int_{\Omega}-\chi(f_j)H_m(f_j) < +\infty.$$

Since $\chi$ is increasing then for all $t>0$
\begin{align*}
&\int\limits_{\{f_j<-t\}}H_m(f_j)\leq\int\limits_{\{f_j<-t\}}\frac{\chi(f_j)}{\chi(-t)}H_m(f_j)\\
& \leq (\chi(-t))^{-1}\sup\limits_{j}\int\limits_{\Omega}\chi(f_j)H_m(f_j).
\end{align*}

Since the sequence $\{f_j<-t\}$ is  increasing  to $\{f< -t\}$ then by letting $j\to\infty$ we get
$$\int\limits_{\{f<-t\}}H_m(f)\leq (\chi(-t))^{-1}\sup\limits_{j}\int\limits_{\Omega}\chi(f_j)H_m(f_j).$$

Now if we let $t\to +\infty$ we deduce that
$$\int\limits_{\{f=-\infty\}}H_m(f)=0.$$

 Hence, $f\in\mathcal{E}_m^a(\Omega)$.\\

$(2)\Rightarrow (1)$  Assume that  $\chi(-\infty)>-\infty$,  then $\mathcal{F}_m(\Omega)\subset\mathcal{E}_{m,\chi}(\Omega)$. But it is known that $\mathcal{F}_m(\Omega)$ is not a subset of $\mathcal{E}_m^a(\Omega)$. We deduce that $\mathcal{E}_{m,\chi}(\Omega)\not\subset\mathcal{E}_m^a(\Omega)$.

\end{proof}

The rest of this section will be devoted to give a connection between the class $\mathcal E_{m,\chi}(\Omega)$ and the $Cap_m-$capacity of sublevels  $Cap_m(\{f<-t\})$. As a consequence we deduce a complete characterization of the class $\mathcal E_{m}^p(\Omega)$ introduced by Lu \cite{L1} in term of the $Cap_m-$capacity of sublevel. For this we introduce the class $\hat{{\mathcal E}}_{m,\chi}(\Omega)$ as follows:

\begin{definition}
$$
\hat{{\mathcal E}}_{m,\chi}(\Omega) :=\left\{ \varphi \in \mathcal{SH}_m^-(\Omega) \, / \,
\int_{0}^{+\infty} t^m \chi'(-t)  Cap_m(\{\varphi<-t\}) dt<+\infty
\right\}.
$$
\end{definition}
The previous class coincides with the class $\hat{{\mathcal E}}_{\chi}(\Omega)$ given by  Benelkourchi,  Guedj, and  Zeriahi \cite{C-Z-Bel}, it suffices to take $m=n$ to recover it.
In the following proposition we cite some properties of $\hat{{\mathcal E}}_{m,\chi}(\Omega)$ and we give a relationship between $\mathcal E_{m,\chi}(\Omega)$ and $\hat{{\mathcal E}}_{m,\chi}(\Omega)$:

\begin{proposition}\label{12}
\begin{enumerate}
\item The classe $\hat{{\mathcal E}}_{m,\chi}(\Omega)$ is  convex.
 \item For every $f\in \hat{{\mathcal E}}_{m,\chi}(\Omega)$ and $ g \in \mathcal{SH}_m^-(\Omega)$, one has that
$\max(f, g )  \in \hat{{\mathcal E}}_{m,\chi}(\Omega)$.
 \item $\hat{{\mathcal E}}_{m,\chi}(\Omega) \subset {\mathcal E}_{m,\chi}(\Omega)$.
 \item If we denote by $\hat{\chi}(t)$ the function defined by $\hat{\chi}(t) := \chi(t/2)$, then
 $$
{\mathcal E}_{m,\chi}(\Omega) \subset \hat{{\mathcal E}}_{m,\hat{\chi}}(\Omega).
$$
  \end{enumerate}
\end{proposition}
\begin{proof}
$1)$ Let $f,g\in \hat{{\mathcal E}}_{m,\chi}(\Omega)$ and $0\leq \alpha\leq 1$. Since we have
   $$
\left\{ \alpha f+(1-\alpha)g <-t \right\} \subset
\left\{ f<-t \right\} \cup  \left\{ g <-t \right\}
$$
then  $f+\alpha g \in \hat{{\mathcal E}}_{m,\chi}(\Omega)$. The result follows.\\

$2)$ The proof of this assertion is obvious.\\

$3)$ Take  $f \in \hat{{\mathcal E}}_{m,\chi}(\Omega)$. It remains to construct a sequence $f_j\in{\mathcal E}_{m}^0(\Omega)$
satisfying $$\int_\Omega -\chi(f_j) \, H_m(f_j)<\infty.$$
For this, we may assume  without loss of generality that $f \leq 0$.
If we  set $f_j:=\max(f,-j)$ then  $f_j\in{\mathcal E}_{m}^0(\Omega)$. Using Lemma \ref{08} we get that
\begin{eqnarray*}
\int_\Omega -\chi (f_j) \, H_m(f_j) &=&
\int_0^{+\infty} \chi'(-t) H_m(f_j)(f_j < -t) dt \\
&\leq& \int_0^{+\infty} \chi'(-t) t^m Cap_m(f<-t) dt
<+\infty. \\
\end{eqnarray*}

It follows that $f \in {\mathcal E}_{m,\chi}(\Omega)$.\\
$4)$  The proof of this assertion follows directly using the same argument as in $3)$ and  the second inequality in Lemma \ref{08} for $t=s$.
\end{proof}

\begin{proposition}\label{13}
 Assume that  for all $t<0$ one has $\chi(t) <0$, then for all  $f\in {\mathcal E}_{m,\chi}(\Omega)$ one has
  $$\displaystyle\limsup_{z\to w } f(z) = 0 , \
\forall w \in \partial \Omega.$$
\end{proposition}
 \begin{proof}

  Since by hypothesis we have for all $t<0$; $\chi(t) <0$ so we can assume, without loss of generality, that the length of the set $\{t > 0; t < t_0\ and \ \chi'(-t)\neq 0\}$ is positive for all $t_0 > 0$. We suppose by contradiction that there is $w_0 \in \partial \Omega$ such that $\displaystyle\limsup_{z\to w_0 } f(z) = \varepsilon < 0$. Then there is a ball $B_0$ centered
at $w_0$ satisfying $B_0 \cap \Omega \subset \{f< \frac{\varepsilon}{2}\}$. If we consider  $(K_j)_j$ to  a sequence of regular compact subsets  so that for all $j$ one has $K_j \subset K_{j+1}$ and $B_0 \cap \Omega = \cup K_j $. Then the extremal function $h_{K_j,\Omega}$ belongs to   ${\mathcal E}_{m}^0(\Omega)$ and decreases to  $h_{E,\Omega}$. It is easy to check that $h_{E,\Omega}\not \in {\mathcal F}_{m}(\Omega)$. By the definition of the class ${\mathcal F}_{m}(\Omega)$ we obtain
$$\sup_j  Cap_m(K_j)=\sup _j \int_\Omega H_m( f_{K_j,\Omega}) =
+\infty.$$
So
 $$Cap_m (B_0 \cap \Omega)=+\infty.$$
 We deduce that
 $$Cap_m (\{f< -s\})=+\infty, \ \forall s \leq -\varepsilon /2,$$
 hence
 $$\int_{0}^{+\infty} t^m \chi'(-t)  Cap_m(\{f<-t\}) dt=+\infty.$$

We get a  contradiction with the fact that ${\mathcal E}_{m,\chi}(\Omega) \subset \hat{{\mathcal E}}_{m,\hat{\chi}}(\Omega)$.
 \end{proof}
\begin{proposition}\label{14}
Assume that  $\chi\not\equiv 0.$ If there exists a sequence $(f_k)\subset{\mathcal E}_{m}^0(\Omega)$ such that
$$\sup_{k\in \mathbb{N}}\int_\Omega -\chi(f_k)H_m(f_k)<\infty,$$
then the function $f:=\displaystyle\lim_{k\rightarrow +\infty} f_k\not\equiv-\infty$ and therefore $f\in{\mathcal E}_{m,\chi}(\Omega)$.
\end{proposition}
\begin{proof}
 Using the hypothesis we observe that the length of the set $\{t > 0; t < t_0\ and \ \chi'(-t)\neq 0\}$ is positive. By lemma \ref{08} we get
 $$s^mCap_m (\{f_k<-2s\})\leq\int_{\{f_k<-s\}} H_m(f_k).$$
 Then
 \begin{eqnarray*}
 \int_{0}^{+\infty} t^m \chi'(-t)  Cap_m(\{f<-t\}) dt &=&\lim_{k\to\infty}\int_{0}^{+\infty} t^m \chi'(-t)  Cap_m(\{f_k<-t\}) dt\\
 &\leq &\lim_{k\to\infty}2^m\int_{0}^{+\infty}  \chi'(-t) \int_{\{f_k<-t\}} H_m(f_k) dt\\
&\leq&2^m \sup_{k\in \mathbb{N}}\int_\Omega -\chi(f_k)H_m(f_k)<\infty. \\
\end{eqnarray*}
Note that in the previous inequality we have used the convergence monotone theorem.
We conclude that $f\not\equiv-\infty$ and therefore $f\in{\mathcal E}_{m,\chi}(\Omega)$.
\end{proof}
\begin{theorem}\label{15}
  Assume that for all $t<0$ one has $\chi(t) <0$. Then
$${\mathcal E}_{m,\chi}(\Omega) \subset {\mathcal N}_{m}(\Omega).$$
\end{theorem}

\begin{proof}
By proposition \ref{12}, it suffices to prove that every maximal function $f \in {\mathcal E}_{m,\chi}(\Omega)$ is identically equal to 0. Take a sequence  $f_j \in{\mathcal E}_{m}^0(\Omega)$ as in the definition of the class ${\mathcal E}_{m,\chi}(\Omega)$. So we obtain using  Lemma \ref{08} that
\begin{eqnarray*}
\int_{0}^{+\infty} \chi'(\frac{-s}{2}) f^m Cap_m(\{f<-s\}) ds &=&\lim_{j\to\infty}\int_{0}^{+\infty}   \chi'(\frac{-s}{2}) s^m Cap_m(\{f_j<-s\}) ds
 \\
&\leq&2^m\lim_{j\to\infty}\int_{0}^{+\infty}  \chi'(-s) \int_{(f_j<-s)} H_m(f_j) ds
 \\
&=&2^m\lim_{j\to\infty}\int_\Omega -\chi(f_j)H_m(f_j). \\
\end{eqnarray*}
Since the maximality of  $f \in {\mathcal E}_{m}(\Omega)$ is equivalent to  $H_m(f) = 0$, we deduce that
$$\lim_{j\to\infty}\int_\Omega -\chi(f_j)H_m(f_j)=0.$$

So $Cap_m(\{f <-s\}) = 0$, $\forall s > 0$. It follows that $f\equiv 0$. The proof of the theorem is completed.
\end{proof}
Now we give a complete characterization of  ${\mathcal E}_{m,\chi}(\Omega)$ in term of ${\mathcal N}_{m}(\Omega)$. We will prove essentially the following result
\begin{corollary}
If for all $t<0$; $\chi(t) <0$ then
 $${\mathcal E}_{m,\chi}(\Omega)=\left\{ f \in {\mathcal N}_m(\Omega) \, / \,
\chi(f) \in L^1(H_m(f)) \right\}.$$
\end{corollary}
\begin{proof}
The first inclusion is a direct deduction from theorem \ref{10} and theorem \ref{15}.
It suffices to prove the reverse inclusion
$$\left\{ f \in {\mathcal N}_m(\Omega) \, / \,
\chi(f) \in L^1(H_m(f)) \right\}\subset{\mathcal E}_{m,\chi}(\Omega).$$
Take $f \in {\mathcal N}_{m}(\Omega)$ satisfying $\int_\Omega -\chi(f)H_m(f)<\infty$. It suffices to construct sequence $f_j \in{\mathcal E}_{m}^0(\Omega)$ that decreases to $f$ and satisfies
 $$\sup_j\int_\Omega -\chi(f_j)H_m(f_j)<\infty.$$
 Let $\rho$ be an exhaustion function for $\Omega$  ($\Omega = \{\rho< 0\}$). The theorem 5.9 in \cite{HP} guarantee that for all $j \in \mathbb{N}$, there is a function $f_j \in{\mathcal E}_{m}^0(\Omega)$ satisfying $H_m(f_j) = 1_{\{f>j\rho\}}H_m(f)$. We have  $H_m(f_{j})\leq H_m(f_{j+1})\leq H_m(f)$, so we get that $f_j\geq f_{j+1}$ using the comparison principle and   $(f_j)_j$converges to a function $\widetilde{f}$. It is easy to check that $\widetilde{f}\geq f$. Now following  the proof of  Theorem \ref{10} we deduce the existence of   a negative m$-$sh function $g$ satisfying $\int_\Omega -g H_m(f)<\infty$. If follows by  Theorem 2.10 \cite{A}
that $\widetilde{f}= f$. Thus the monotone convergence theorem gives
$$\int_\Omega -\chi(f_j)H_m(f_j)=\int_\Omega -\chi(f_j)1_{\{f>j\rho}\}H_m(f)\rightarrow\int_\Omega -\chi(f)H_m(f)<\infty.$$
 \end{proof}
Now we will extend the theorem A to the class $\mathcal{E}_{m,\chi}(\Omega)$.
\begin{theorem}\label{11}

  Assume that $\chi$ is  continuous, $\chi(-\infty)>-\infty$ and   $f, f_j\in \mathcal{E}_{m}(\Omega)$ for all $j\in \mathbb{N}$. If there exists $g\in \mathcal{E}_{m}(\Omega)$ satisying $f_j\geq g$ on $\Omega$ then:
  \begin{enumerate}
    \item If $f_j$ converges to $f$ in $Cap_{m-1}-$capacity then $\displaystyle\liminf_{j\rightarrow +\infty}-\chi(f_j)H_m(f_j)\geq -\chi(f)H_m(f).$
    \item If $f_j$ converges to $f$ in $Cap_{m}-$capacity then $-\chi(f_j)H_m(f_j)$ converges weakly to $-\chi(f)H_m(f).$
  \end{enumerate}
\end{theorem}
\begin{proof}
$(1)$ Take a test function $\varphi\in C^{\infty}_{0}(\Omega)$ such that $0\leq \varphi\leq 1$. Using \cite{L2} there exist $\psi_k\in \mathcal{E}^0_{m}(\Omega)\cap \mathcal{C}(\Omega)$ with $\psi_k\geq f$ and $\psi_k\searrow f$ in $\Omega$. For a fixed integer $k\geq 1$ there exists, by  \cite{H},  $j_0\in \mathbb{N}$ such that $f_j\geq \psi_k$ on $supp\ \varphi$ for all $j\geq j_0$. So by Theorem 3.10 in \cite{HP}, we obtain that for all $k\geq 1$ one has
$$ \displaystyle\liminf_{j\rightarrow +\infty}\int_{\Omega}-\varphi\chi(f_j)H_m(f_j)\geq \displaystyle\liminf_{j\rightarrow +\infty}\int_{\Omega}-\varphi\chi(\psi_k)H_m(f_j)= \int_{\Omega}-\varphi\chi(\psi_k)H_m(f).$$
Now if we let $k$ tends to $+\infty$ then by the Lebesgue monotone convergence theorem, we get
  $$ \displaystyle\liminf_{j\rightarrow +\infty}\int_{\Omega}-\varphi\chi(f_j)H_m(f_j)\geq\int_{\Omega}-\varphi\chi(f)H_m(f).$$
  The result follows.\\
  $(2)$ Without loss of generality one can assume that $\chi(-\infty)=-1$. Let $\varphi\in C^{\infty}_{0}(\Omega)$ such that $0\leq \varphi\leq 1$.
  We claim that
  $$\displaystyle\limsup_{j\rightarrow +\infty}\int_{\Omega}-\varphi\chi(f_j)H_m(f_j)\leq\int_{\Omega}-\varphi\chi(f)H_m(f).\qquad (*)$$
  Indeed, by the quasicontinuity of $f$ and $g$ with respect to the capacity $Cap_m$, we obtain that for every $k\in \mathbb{N}$ there exist an open subset $O_k$ of $\Omega$ and a function $\widetilde{f}_k\in \mathcal{C}(\Omega)$ such that $Cap_m(O_k)\leq \frac{1}{2^k}$ and $\widetilde{f}_k=f$ on $\Omega \setminus O_k$ and $g\geq -\alpha_k$ on $supp \varphi \setminus O_k$ for some $\alpha_k>0$.
  Let $\varepsilon >0$, then by Theorem 3.6 in \cite{Hi}  one has
  $$\begin{array}{lcl}
     \displaystyle\int_{\Omega}-\varphi\chi(f_j)H_m(f_j)& = & \displaystyle\int_{\Omega \setminus O_k}-\varphi\chi(f_j)H_m(f_j)  + \displaystyle\int_{O_k}-\varphi\chi(f_j)H_m(f_j) \\
    & \leq & \displaystyle\int_{\Omega \setminus O_k}-\varphi\chi(f_j)H_m(f_j)  + \displaystyle\int_{O_k}-\varphi H_m(f_j)\\
    & \leq & \displaystyle\int_{\{f_j\leq f-\varepsilon\} \setminus O_k}-\varphi\chi(f_j)H_m(f_j)\\
     &+&\displaystyle\int_{\{f_j> f-\varepsilon\} \setminus O_k}-\varphi\chi(f_j)H_m(f_j)+ \displaystyle\int_{O_k}-\varphi H_m(f_j)\\
     & \leq & \displaystyle\int_{\{f_j\leq f-\varepsilon\} \setminus O_k}-\varphi H_m(f_j)\\
     &+&\displaystyle\int_{\Omega \setminus O_k}-\varphi\chi(f-\varepsilon)H_m(f_j) + \displaystyle\int_{\Omega}-\varphi h_{O_k,\Omega} H_m(f_j)\\
     & \leq & \displaystyle\int_{\{f_j< f-\varepsilon\} \setminus O_k} H_m(\max(f_j, -\alpha_k))\\
     &+&\displaystyle\int_{\Omega \setminus O_k}-\varphi\chi(\widetilde{f}_k-\varepsilon)H_m(f_j) + \displaystyle\int_{\Omega}-\varphi h_{O_k,\Omega} H_m(f_j)\\
      & \leq & \alpha_k^m Cap_m({\{f_j< f-\varepsilon\}\cap supp \varphi})\\
     &+&\displaystyle\int_{\Omega \setminus O_k}-\varphi\chi(\widetilde{f}_k-\varepsilon)H_m(f_j) + \displaystyle\int_{\Omega}-\varphi h_{O_k,\Omega} H_m(f_j).\\

  \end{array}
  $$
  If we let $j$ goes to $+\infty$, we get using theorem 3.8 \cite{HP} that
  $$\begin{array}{lcccl}
     \displaystyle\limsup_{j\rightarrow +\infty}\int_{\Omega}-\varphi\chi(f_j)H_m(f_j)& \leq &\displaystyle\int_{\Omega \setminus O_k}-\varphi\chi(\widetilde{f}_k-\varepsilon)H_m(f)& + &\displaystyle\int_{\Omega}-\varphi h_{O_k,\Omega} H_m(f) \\
  \end{array}
  $$
  If we let $\varepsilon\rightarrow 0$, we obtain
  $$\begin{array}{lcccl}
     \displaystyle\limsup_{j\rightarrow +\infty}\int_{\Omega}-\varphi\chi(f_j)H_m(f_j)& \leq &\displaystyle\int_{\Omega \setminus O_k}-\varphi\chi(\widetilde{f}_k) H_m(f)& + &\displaystyle\int_{\Omega}-\varphi h_{O_k,\Omega} H_m(f) \\
     &\leq &\displaystyle\int_{\Omega \setminus \{f=-\infty\}}-\varphi\chi(f)H_m(f)& + &\displaystyle\int_{\Omega}-\varphi h_{\bigcup_{l=k}^{\infty}O_l,\Omega} H_m(f) \qquad (**) \\
  \end{array}
  $$
  Now as $\bigcup_{l=k}^{\infty}O_l\searrow O$ when $k\longrightarrow +\infty$ then
  $$Cap_m(O)\leq\lim_{k\longrightarrow \infty}Cap_m\left(\displaystyle\bigcup_{l=k}^{\infty}O_l\right)\leq \lim_{k\longrightarrow \infty}\sum_{l=k}^{\infty}Cap_{m}(O_l)\leq \lim_{k\longrightarrow \infty}\frac{1}{2^{k-1}}$$
  so there exists an $m-$polar set $M$ such that   $h_{\bigcup_{l=k}^{\infty}O_{l,\Omega}}\nearrow 0$ when $k\longrightarrow +\infty$ on $\Omega\setminus M$. So if we take  $k\longrightarrow +\infty$ in $(**)$, we obtain
  $$\begin{array}{lcccl}
     \displaystyle\limsup_{j\rightarrow +\infty}\int_{\Omega}-\varphi\chi(f_j)H_m(f_j)&\leq &\displaystyle\int_{\Omega \setminus \{f=-\infty\}}-\varphi\chi(f)H_m(f)& + &\displaystyle\int_{M}\varphi  H_m(f)\\
      &\leq &\displaystyle\int_{\Omega \setminus \{f=-\infty\}}-\varphi\chi(f)H_m(f)& + &\displaystyle\int_{\{f=-\infty\}}-\varphi\chi(f)  H_m(f)\\
      &=&\displaystyle\int_{\Omega}-\varphi\chi(f)  H_m(f).
  \end{array}
  $$
  This proves the claim $(*)$. Moreover since $f_j$ converges in $Cap_{m}-$capacity so it converges in  $Cap_{m-1}-$capacity. Using the assertion $(a)$ we obtain
  $$\liminf_{j\rightarrow +\infty}\int_{\Omega}-\varphi\chi(f_j)  H_m(f_j)\geq \int_{\Omega}-\varphi\chi(f)  H_m(f).$$
  If we combine the last inequality with $(**)$ we get
  $$\lim_{j\rightarrow +\infty}\int_{\Omega}-\varphi\chi(f_j)  H_m(f_j)= \int_{\Omega}-\varphi\chi(f)  H_m(f),$$
  for every $\varphi \in \mathcal{C}_{0}^{\infty}(\Omega)$ with $0\leq \varphi \leq 1$. Hence we get the desired result.
  \end{proof}

 Now we will be intrusted to the problem of subextention in the class ${\mathcal E}_{m,\chi}(\Omega)$. For  $\Omega\Subset\tilde{\Omega}\Subset \mathbb{C}^n$ and $f\in {\mathcal E}_{m,\chi}(\Omega)$, we say that $\tilde{f}\in {\mathcal E}_{m,\chi}(\tilde{\Omega})$ is a subextention of $f$ if $\tilde{f} \leq f$ on $\Omega$. In the following theorem we prove that every function $f\in {\mathcal E}_{m,\chi}(\Omega)$ has a subextention.

\begin{theorem}
  Let $\tilde{\Omega}$ be a m$-$hyperconvex domain such that $\Omega\Subset\tilde{\Omega}\Subset \mathbb{C}^n$. If  $\chi(t) <0$ for all $t<0$  and $f\in {\mathcal E}_{m,\chi}(\Omega)$  then is $\tilde{f}\in {\mathcal E}_{m,\chi}(\tilde{\Omega})$
satisfying
$$\int_{\tilde{\Omega}} -\chi(\tilde{f})H_m(\tilde{f})\leq\int_\Omega -\chi(f)H_m(f)$$
and $\tilde{f} \leq f$ on $\Omega$.
\end{theorem}
\begin{proof}
  Let $f\in {\mathcal E}_{m,\chi}(\Omega)$ and  $f_k\in {\mathcal E}_{m}^0(\Omega)$ be the  sequence as in the definition of the class ${\mathcal E}_{m,\chi}(\Omega)$. We obtain using lemma 3.2 in \cite{Maw}  that for every  $k \in \mathbb{N}$, there exists a subextension $\tilde{f}_k$ of $f_k$. It follows that
  \begin{eqnarray*}
\int_{\tilde{\Omega}} -\chi(\tilde{f_k})H_m(\tilde{f_k}) &=&\int_{\{\tilde{f}_k=f_k\}\cap\Omega} -\chi(\tilde{f_k})H_m(\tilde{f_k})
 \\
&\leq&\int_{\{\tilde{f}_k=f_k\}\cap\Omega} -\chi(f_k)H_m(f_k) \\
&\leq&\int_{\Omega} -\chi(f_k)H_m(f_k). \\
\end{eqnarray*}
So we obtain
$$\sup_k\int_{\tilde{\Omega}} -\chi(\tilde{f_k})H_m(\tilde{f_k})\leq\int_{\Omega} -\chi(f)H_m(f)<\infty.\quad (*)$$
Using the proposition \ref{14} we get that the function $\tilde{f} = \lim_{k\rightarrow\infty} \tilde{f}_k\not\equiv -\infty$ and $\tilde{f}\in {\mathcal E}_{m,\chi}(\tilde{\Omega})$. Then by $(*)$
$$\int_{\tilde{\Omega}} -\chi(\tilde{f})H_m(\tilde{f})\leq\int_{\Omega} -\chi(f)H_m(f)<\infty.$$
It follows by the Comparison Principle that for all $k \in \mathbb{N}$ one has $\tilde{f}_k \leq f_k$ on $\Omega$. If we let $k$
goes to $\infty$, we deduce that $\tilde{f} \leq f$ on $\Omega$.
\end{proof}

$\mathbf{Acknowledgments}$
Authors extend their appreciation to the Deanship of Scientific Research at Jouf University for funding this work through research Grant no. DSR-2021-03-03134.\\

\end{document}